\documentclass[11pt,reqno]{amsart}
\usepackage[ansinew]{inputenc}
\usepackage[spanish,english]{babel}
\usepackage{float}
\usepackage{graphics}
\usepackage{amssymb}
\usepackage{amscd}
\newcommand{\K}{\mathcal{K}}
\newcommand{\h}{\mathcal{H}}
\newcommand{\Complex}{\mathbb{C}}

 \newcommand{\To}{\longrightarrow}

 \newcommand{\N}{\mathbb{N}}

 \newcommand{\ip}[2][\cdot\hs]{\langle #1,#2\rangle}

 \newcommand{\hs}{{\hspace{1pt}}}

\renewcommand{\[}{\begin{equation}}
\renewcommand{\]}{\end{equation}}
  
   \newcommand{\fN}{\mathfrak{N}}
\newtheorem{theorem}{Theorem}[section]

\newtheorem{corollary}[theorem]{Corollary}
\newtheorem{proposition}[theorem]{Proposition}
\newtheorem{definition}[theorem]{Definition}

\newtheorem{remark}[theorem]{Remark}

\newtheorem{example}[theorem]{Example}
\numberwithin{equation}{section}

\begin{document}

%\selectlanguage{spanish}

\title[***]
{Construction, Extension and Coupling of Frames on Finite Dimensional Pontryagin Space. \\
***}

 \author{Escobar German }
\address{Escobar German, Programa  de Matem\'atica aplicada, Fac. Ciencias Exactas y Naturales, Universidad Surcolombiana, Neiva, Colombia.}
\email{german.escobar@usco.edu.co}

\author[Esmeral Kevin]{Esmeral Kevin}
\address{Esmeral Kevin. Departamento de Matem\'aticas \\Centro de investigaci\'on y de estudios avanzados del Instituto Polit\'ecnico Nacional\\(CINVESTAV-IPN), M\'exico.}
 \email{kmesmeral@math.cinvestav.mx, matematikoua@gmail.com}

\author{Ferrer Osmin}
\address{O. Ferrer, Programa  de Matem\'atica, Fac. Educaci\'on, Universidad Surcolombiana, Neiva, Colombia. }
\email{osmin.ferrer@usco.edu.co}

%    General info
% Dual algebras; weakly closed singly generated operator algebras
% Banach algebras of continuous functions

\date{\today}

%\commby{XXX}


\begin{abstract}

\noindent In this paper we extend to finite-dimensional Pontryagin spaces the methods used in \cite{CasazzaLeon,Deguang} to build frames from an adjoint and positive operator. It is proved
that any frame in finite dimensional Pontryagin space $\K$ is $J$-orthogonal projection of a frame for a space $\h$ such that $\K\subset\h$. Furthermore, given
$\{k_{n}\}_{n=1}^{m}$ and $\{x_{n}\}_{n=1}^{k}$ frames for $\K$ and $\h$ respectively, we build a finite-dimensional Pontryagin space $\Re$ and a frame $\{y_{n}\}_{n=1}^{N}$ for $\Re$ such that $\K,\h\subset\Re$ and $\{k_{n}\}_{n=1}^{m},\{x_{n}\}_{n=1}^{k}\subset \{y_{n}\}_{n=1}^{N}$.
\end{abstract}

%%% ----------------------------------------------------------------------
\maketitle
%%% ---

%\noindent \small{\textit{R. Bruzual. Universidad Central de Venezuela.} ramonbruzual@gmail.com \\
%\textit{M. Dom\'{\i}nguez. Universidad Central de Venezuela.} dominguez.math@gmail.com \\
%\textit{O. Ferrer. Universidad Surcolombiana} osmin.ferrer@usco.edu.co}
%\setlength{\baselineskip}{18pt}

\section*{Introduction}

From its appearance  in \cite{DS} the theory of frames in Hilbert space has been quickly developed \cite{Casazza, CasazzaLeon, Dau,Deguang,G, RNG}, unlike the theory of frames in Krein space which is giving its firts steps, \cite{POK, KEFER, GMMM, GMMMa, PW}. In \cite{KEFER} a family $\{k_{n}\}_{n\in\N}$  is a frame for a Krein space if there exist constants $A,B>0$  such that

\begin{equation*}
  A\|k\|_{J}^{2}\leq\sum_{n\in\N}|[k,k_{n}]|^{2}\leq\,B\|k\|_{J}^{2},\quad\forall k\in\mathcal{K},
\end{equation*}

In \cite{GMMM} and \cite{PW}  alternative definitions are proposed. The fundamental idea is to use the versatility and flexibility of frames.\\
In \cite{CasazzaLeon} and \cite{Deguang} we find methods to build and extend frames in finite-dimensional Hilbert space. Based on \cite{KEFER} the main purpose of this work
is to understand and extend these results to finite-dimensional Krein space called Pontryagin space. It is further proved that if $\{k_{n}\}_{n=1}^{m}$ y $\{x_{n}\}_{n=1}^{k}$ are frames
for the Krein space $\K$ y $\h$ respectively then it is possible coupling those families, where by coupling it will be understood, find a Krein space $\Re$ con $\K,\h\subset\Re$ and a frame  $\{y_{n}\}_{n\in\N}$ such that $\{k_{n}\}_{n=1}^{m},\{x_{n}\}_{n=1}^{k}\subset \{y_{n}\}_{n=1}^{N}$.

To achieve our objective, in section \ref{preliminares}  basic aspects of the finite-dimensional Pontryain space and the theory of frames in fnite-dimensional krein space are given.
In section \ref{resultadosprincipales} we find the main results of this paper, we build finite frames from a positive and adjoint operator in the Pontryagin space
(subsection \ref{construccion}). We extend frames to bigger spaces (subsection \ref{extension}), and finally we couple two frames keep on the condition of frame. That is, we build a
frame for the Pontyagrin space which contains both frames (subsection \ref{acoplamiento}).

\section{Preliminaries}\label{preliminares}

\begin{definition}[\textit{Krein Spaces}]
Let $\mathcal{K}$ be a vector space on $\Complex$. Consider $[\cdot,\cdot]:\,\mathcal{K}\times\mathcal{K}\longrightarrow\Complex$,  a  sesquilinear form.  The vector space $\left(\mathcal{K},\,[\cdot,\cdot]\right)$ is a \emph{Krein space} whether $\mathcal{K}=\mathcal{K}^{+}\oplus\mathcal{K}^{-}$ and  $\left(\mathcal{K}^{+},\,[\cdot,\cdot]\right)$, $\left(\mathcal{K}^{-},\,-[\cdot,\cdot]\right)$ are Hilbert spaces, where $\mathcal{K}^{+}$, $\mathcal{K}^{-}$ are orthogonal with respect $[\cdot,\cdot]$.
\end{definition}
On $\mathcal{K}$ define the following scalar product
$$(x_{1},x_{2})=[x_{1}^{+},x_{2}^{+}]-[x_{1}^{-},x_{2}^{-}],\; x_{i}^{\pm}\in\mathcal{K}^{\pm},\; x_{i}=x_{i}^{+}+x_{i}^{-}.$$
This scalar product makes $\left(\mathcal{K},\,(\cdot,\cdot)\right)$  a Hilbert space, which is so-called \emph{Hilbert space associated to }$\mathcal{K}.$ Hence, we can take the orthogonal projections on $\mathcal{K}^{+}$ and $\mathcal{K}^{-}$ denoted $P^{+}$ and $P^{-}$ respectively. The linear bounded operator $J=P^{+}-P^{-}$ is called \emph{Fundamental Symmetry}, and it satisfies the equality $[x,y]=(Jx,y),\;\forall x,\,y\in\mathcal{K}.$ Equivalently
\begin{equation}\label{jnorma}
[x,y]_{J}=[Jx,y]=(x,y),\quad\text{and denote} \quad\|x\|_{J}=\sqrt{[x,x]_{J}}\,\forall x,\,y\in\mathcal{K}.
\end{equation}
\begin{definition}
Let $(\mathcal{K},[\cdot,\cdot])$ be a Krein space. Consider $x,y\in\mathcal{K}$,  we say that $x$ is orthogonal to $y$ if $[x,y]_{J}=0$, and is denoted by $x\perp y$. We say that $x$ is $J$-\emph{orthogonal} to $y$  if $[x,y]=0$, and is denoted by $x[\perp]y$.
\end{definition}
The main purpose in this paper is to study frames in  $N$-dimensional Krein spaces $\K$.  Therefore $\operatorname{dim}\K^{+},\operatorname{dim}\K^{-}\leq N$, these spaces are called \emph{Pontryagin spaces}. More details see \cite{Azizov, Bognar}.
\begin{definition}\cite{Azizov, Bognar}
  A Krein space $(\K,[\cdot,\cdot])$ such that
  \begin{equation}
   0< \aleph=\min\{\operatorname{dim}\K^{+},\operatorname{dim}\K^{-}\}<+\infty
  \end{equation}
 is called a \emph{Pontryagin space}.
\end{definition}

\begin{definition}\cite{Azizov, Bognar}
Let $\mathcal{K}$ be a Krein space, and let $V$ be a closed subspace of $\mathcal{K}$. The subspace
  \begin{equation}\label{Vperp}
V^{[\perp]}=\left\{x\in\mathcal{K}: [x,y]=0,\,\text{for all $y\in\,V$}\right\}
\end{equation}
 is  so-called the $J$-\emph{orthogonal complement of $V$ with respect $[\cdot,\cdot]$} (or simply $J$-\emph{orthogonal complement} of $V$).
\end{definition}
\begin{definition}\cite{Azizov, Bognar}
A closed subspace $V$ of $\mathcal{K}$  such that $V\cap V^{[\perp]}=\{0\}$ and $V+ V^{[\perp]}=\mathcal{K}$, where $V^{[\perp]}$ is given in \eqref{Vperp}
 is called \emph{projectively complete}.
\end{definition}
\begin{proposition}(\cite{Azizov,Bognar}) Let $(\K,[\cdot,\cdot])$ be a Pontryagin space $k$-dimensional. Every  closed subspace $V\subset\K$ is projectively complete. Furthermore, is itself a Pontryagin space with dimension $0\leq\,k^{\prime}\leq\,k$.
\end{proposition}
\begin{remark}
 In \cite{Azizov} shown's that for any closed subspace $V$   its $J$-orthogonal complement $V^{[\perp]}$ and its orthogonal complement $V^{\perp}$ are closed subspaces connected by the formulas
  \begin{align}
    V^{[\perp]}=JV^{\perp},\quad   V^{\perp}=JV^{[\perp]}, \quad   (JV)^{[\perp]}=JV^{[\perp]}\label{subspacepkpk}.
  \end{align}
  By \eqref{subspacepkpk} we note that $JV$  is projectively complete if and only if $V$ is projectively complete. In addition, the condition $V\cap V^{[\perp]}=\{0\}$ tells that every $x\in\mathcal{K}$ has an unique  $J$-orthogonal projection on $V$, see \cite{Azizov,Bognar}.
\end{remark}
\begin{remark}\label{remarkJVcV}
Let $(\mathcal{K},[\cdot,\cdot])$ be a Krein space. Let $V$ be a closed subspace of $\mathcal{K}$ which is projectively complete. Then, $V=\overline{V}=\left(V^{[\perp]}\right)^{[\perp]}$, this implies that $(V,[\cdot,\cdot])$ is a Krein space. Hence $V=V^{+}[\dotplus]\,V^{-}$, where $V^{+}\subset\mathcal{K}^{+}$, and $V^{-}\subset\mathcal{K}^{-}$. Thus $JV\subset V$.
\end{remark}
\begin{definition}\cite{Azizov,Bognar}
  Let $\mathcal{K}$ be a Krein space. A system vector $\{e_{i}\}_{i\in I}\subset \mathcal{K}$, where $I$ is an arbitrary set of indices is called a $J$-\emph{orthonormalized system} if $[e_{i},e_{j}]=\pm\delta_{i,j}$ for all $i,j\in I$, where $\delta_{i,j}$ is the Kronecker delta.
\end{definition}
\begin{example}
   The simplest example of an $J$-orthonormalized system of $\mathcal{K}$ is  the union of two arbitrary orthonormalized (in the usual sense) systems from the subspaces $\mathcal{K}^{+}$ and $\mathcal{K}^{-}$ respectively.
\end{example}
\begin{definition}\cite{Azizov,Bognar}
A $J$-orthonormalized system is said to be \emph{maximal} if it is not contained in any wider $J$-orthonormalized system, and to be $J$-\emph{complete} if there is no non-zero vector $J$-orthonormalized to this system.
\end{definition}
\begin{definition}\cite{Azizov,Bognar}
  Let $(\mathcal{K},[\cdot,\cdot])$ be a Krein space. An $J$-\emph{orthonormalized basis} in $\mathcal{K}$ is an  $J$-orthonormalized system which is $J$-complete and maximal in $\mathcal{K}$.
\end{definition}
\begin{theorem}(\cite{Azizov,Bognar})\label{teoremabasekrein}
  let $\{e_{n}\}_{n\in\N}$ be $J$-orthonormal system in the Krein space $\left(\K,[\cdot,\cdot]\right)$. The following conditions are equivalent.
  \begin{enumerate}
    \item [$i).$] $\K$ admits a fundamental decomposition $\K=\K^{+}[\dotplus]\K^{-}$ such that the $e_{n}$'s with $[e_{n},e_{n}]=1$  belong to $\K^{+}$, and the $e_{n}$'s with $[e_{n},e_{n}]=-1$ belong to $\K^{-}$.
     \item [$ii).$]   \begin{equation}\label{igualdaden}
      x=\sum_{n\in\N}[e_{n},e_{n}][x,e_{n}]e_{n},\quad\forall\,x\in\K.
      \end{equation}
  \end{enumerate}
\end{theorem}
\begin{example}[ $\Complex^{k}$ like a Pontryagin space]\label{examplecomplexk}
\begin{equation}
  \Complex^{k}=\left\{(z_{1},z_{2},\ldots,z_{k}): z_{n}\in\Complex,\quad n=1,2,\ldots,k\right\}.
\end{equation}
  On it, we define the indefinite inner product
  \begin{equation}\label{prodindecomplex}
    [z,k]_{\Complex^{k}}:=\sum_{n=1}^{k}(-1)^{n-1}z_{n}\overline{w_{n}},\quad z,w\in\Complex^{k}.
  \end{equation}
  Now,  $\Complex^{k}=\Complex_{+}^{k}+\Complex_{-}^{k}$, where
  \begin{align}
\Complex_{+}^{k}&:=\left\{z\in\Complex^{k}:z=\left(z_{1},0,z_{3},\ldots,z_{k}\right),\,\text{if  $k=2n-1$ for some $n\in\N$}\right\}\\
\Complex_{-}^{k}&:=\left\{z\in\Complex^{k}:z=\left(0,z_{2},0,z_{4},\ldots,z_{k}\right),\,\text{if  $k=2n$ for some $n\in\N$}\right\}.
  \end{align}
 Observe that if $z\in\Complex_{+}^{k}$ and $w\in\Complex_{-}^{k}$, then
  \begin{align*}
[z,z]_{\Complex^{k}}&=\sum_{n=1}^{k}(-1)^{n-1}|z_{n}|^{2}=\sum_{\text{$n$ is even}}^{k}|z_{n}|^{2}>0\\
[w,w]_{\Complex^{k}}&=\sum_{n=1}^{k}(-1)^{n-1}|w_{n}|^{2}=-\sum_{\text{$n$ is odd}}^{k}|w_{n}|^{2}<0.
\end{align*}
On the other hand, if $z\in\Complex_{+}^{k}$ and $w\in\Complex_{-}^{k}$, then
  $$[z,w]_{\Complex^{k}}=\sum_{n=1}^{k}(-1)^{n-1}z_{n}\overline{w_{n}}=0.$$
Define the linear operator $\widetilde{J}:\Complex^{k}\To\Complex$ given by
\begin{equation}\label{Jcomplexk}
\widetilde{J}z=(z_{1},-z_{2},z_{3},\ldots,(-1)^{k-1}z_{k}),\quad z=(z_{1},z_{2},\ldots,z_{k})\in\Complex^{k}.
\end{equation}
This operator is self-adjoint, $J$-self-adjoint with $\widetilde{J}^{2}=\mathrm{id}$, and
\begin{equation}
  [\widetilde{J}z,w]_{\Complex^{k}}=\left[\left(z_{1},-z_{2},\ldots,(-1)^{k-1}z_{k}\right),\left(w_{1},w_{2},\ldots,w_{k}\right)\right]=\sum_{n=1}^{k}z_{n}\overline{w_{n}}=\langle\,z,w\rangle_{\Complex^{k}}.
\end{equation}
   Hence, $\mathfrak{C}^{k}=\left(\Complex^{k},[\cdot,\cdot]_{\Complex^{k}}\right)$ is a Pontryagin space with fundamental symmetry $\widetilde{J}$.
\end{example}

Since  $(\mathcal{K},[\cdot,\cdot]_{J})$ is a Hilbert space, we can study linear operators  acting on Krein spaces. The topological concepts as continuity, closedness operators and spectral theory and so on, refer to the topology induced by the $J$-norm given in \eqref{jnorma}. Therefore, we can concluded that some definitions of \emph{operator theory in Hilbert spaces} are satisfied.  The adjoint of an operator $T$ in Krein spaces $\left(T^{[\ast]}\right)$ satisfies $[T(x),y] = [x,T^{[\ast]}(y)]$, but, we must consider that $T$ have an adjoint operator in the Hilbert space  $(\mathcal{K}, [\cdot,\cdot]_{J})$  denoted $\left(T^{\ast J}\right)$, where  $J$  is the fundamental symmetry in $\mathcal{K}$, and there is a relation between $T^{\ast J}$ and $T^{[\ast]}$, which is $T^{[ \ast ] } = J T^{\ast J}J$. Moreover, let $\mathcal{K}$ and $\mathcal{K}'$ be Krein spaces  with fundamental symmetries $J_{\mathcal{K}}$ and $J_{\mathcal{K}'}$ respectively, if $T\in\mathcal{B}(\mathcal{K},\mathcal{K}')$ then $T^{[*]_{\mathcal{K}}}=J_{\mathcal{K}}T^{*J_{\mathcal{K}}}J_{\mathcal{K}'}$. An operator $T\in\mathcal{B}(\mathcal{K})$ is said to be \emph{self-adjoint} if $T=T^{[*]}$, and $J$-\emph{self-adjoint} whether $T=T^{*J}$, moreover, a linear operator $T$ is said to be \emph{positive} whether $[Tx,x]\geq0$ for every $x\in\mathcal{K}$. An operator $T$ is said to be \emph{uniformly positive} if there exists $\alpha>0$ such that
$[Tx,x]\geq\alpha\|x\|_{J}$ for every $k\in\mathcal{K}$. Furthermore, a linear operator $T$ is said to be \emph{invertible} when its range and domain are the whole space.
\begin{proposition}\label{proposicionproductokrein}
 Let $\left(\mathcal{K},[\cdot,\cdot]_{\K}\right),  \left(\mathcal{H},[\cdot,\cdot]_{\h}\right)$ be Krein spaces with fundamental symmetries $J_{\mathcal{K}}$,  $J_{\mathcal{H}}$  respectively. The vector space $\Re=\K\times\h$ with sesquilinear form
 \begin{equation}
 [\cdot,\cdot]_{\Re}:=[\cdot,\cdot]_{\K}+[\cdot,\cdot]_{\h},
 \end{equation}
is a Krein space with fundamental symmetry
\begin{equation}
  J_{\Re}=\left(J_{\K},J_{\h}\right).
\end{equation}
\end{proposition}
\begin{proof}
Since $\K=\K^{+}[\dotplus]\K^{-}$ and $\h=\h^{+}[\dotplus]\h^{-}$, we define
$$\Re^{+}=\K^{+}\times\h^{+},\quad\Re^{-}=\K^{-}\times\h^{-}.$$
Thus, $\Re^{\pm}\subset\Re$, and $\Re=\Re^{+}+\Re^{-}$. Observe that if $(k^{+},h^{+})\in\Re^{+}$, and  $(k^{-},h^{-})\in\Re^{-}$, then
$$[(k^{+},h^{+}),(k^{-},h^{-})]_{\Re}=[k^{+},k^{-}]_{\K}+[h^{+},h^{-}]_{\h}=0.$$
Equivalently, $\Re^{+}[\perp]\Re^{-}$. On the other hand, $\left(\Re^{+},[\cdot,\cdot]_{\Re}\right)$ and $\left(\Re^{-},-[\cdot,\cdot]_{\Re}\right)$ are Hilbert spaces. In fact, if $\{x_{n}\}_{n\in\N}$ is a Cauchy sequence in $\left(\Re^{+},[\cdot,\cdot]_{\Re}\right)$, then $\{x_{n}\}_{n\in\N}$ converges in $\left(\Re^{+},[\cdot,\cdot]_{\Re}\right)$ if and only if $\{P_{\K^{+}}x_{n}\}_{n\in\N}$ converges in $\left(\K^{+},[\cdot,\cdot]_{\K}\right)$ and $\{P_{\h^{+}}x_{n}\}_{n\in\N}$ converges in $\left(\h^{+},[\cdot,\cdot]_{\h}\right)$, where $P_{\K^{+}}$ and $P_{\h^{+}}$ are the orthogonal projection on $\K^{+}\times\{0\}$ and $\{0\}\times\h^{+}$ respectively. Therefore, $\left(\Re,[\cdot,\cdot]_{\Re}\right)$ is a Krein space.
Now, consider $  J_{\Re}=\left(J_{\K},J_{\h}\right):\Re\To\Re$, it satisfies the following properties
\begin{align*}
[J(x,y),(a,b)]_{\Re}&=[(J_{\K}x,J_{\h}y),(a,b)]_{\Re}=[J_{\K}x,a]_{\K}+[J_{\h}y,b]_{\h}\\
&=[x,J_{\K}a]_{\K}+[y,J_{\h}b]_{\h}=[(x,y),J(a,b)]_{\Re}.
\end{align*}
Thus $J^{[*]}=J$. Also, $J^{2}\left(x, y\right) = J\left(J_{\K}x, J_{\h}y\right)=\left(J^{2}_{\mathcal{K}}x,  J^{2}_{\mathcal{H}}y\right)= \left(x, y\right).$ i.e $J^{2}=\mathrm{id}$.
\end{proof}
\begin{remark}
On the product space $\Re$ given above, we can define the scalar product
\begin{equation}
  [(x, y),(a,b)]_{J}:=[J(x, y),(a,b)]_{\Re}=[J_{\K}x,a]_{\K}+[J_{\h}y,b]_{\h}.
  \end{equation}
\end{remark}

\subsection{Frames In Krein Spaces}
This subsection is based in the results about the frame theory in Krein spaces studied in \cite{KEFER}. We used such results in Pontryagin spaces with dimension $N$.
\begin{definition}  \label{fK}
 Let $\mathcal{K}$ be a Krein space. A countable sequence $\left\{x_{n}\right\}_{n\in \fN}\subset \mathcal{K}$
 is called a \emph{frame for} $\mathcal{K}$, if there exist constants $0< A\leq B<\infty$ such that
\begin{equation}
A\hs \|x\|_{J}^{2}
\,\leq\, \sum_{n\in \fN}\left|\left[x_{n},x\right]\right|^{2}
\,\leq\, B\hs \|x\|_{J}^{2}\quad \text{for all} \  \,x\in\mathcal{K}.  \label{m1}
\end{equation}
\end{definition}
\begin{remark}
Since we are mostly interested in finite-dimensional spaces, and
since one can always fill up a finite frame with zero elements,
we assume that $\fN=\{1,\ldots,k\}$.
\end{remark}

As in the Hilbert space case, we refer to $A$ and $B$ as frame bounds.
The greatest constant $A$ and the smallest constant $B$ satisfying \eqref{m1}
are called optimal lower frame bound and optimal upper frame bound, respectively.
A frame is tight, if one can choose $A=B$. If a frame ceases to be a frame
when an arbitrary element is removed, the frame is said to be exact.

The next theorem shows that frames for a Krein space are essentially the same objects as frames
for the associated Hilbert space.
\begin{theorem}(\cite{KEFER})\label{prop1}
Let $\mathcal{K}$ be a finite-dimensional Pontryagin space and $\{x_{n}\}_{n=1}^{k}$ a sequence in $\mathcal{K}$.
The following statements are equivalent:\\[6pt]
\phantom{ii}i)   $\{x_{n}\}_{n=1}^{k}$ is a frame for the Pontryagin space $\mathcal{K}$ with frame bounds $A\leq B$.\\[6pt]
\phantom{i}ii) $\{Jx_{n}\}_{n=1}^{k}$ is a frame for the  Pontryagin space $\mathcal{K}$ with frame bounds $A\leq B$.\\[6pt]
iii) $\{x_{n}\}_{n=1}^{k}$ is a frame for the  Hilbert space $\left(\mathcal{K}, [\cdot,\cdot]_{J}\right)$ with frame bounds $A\leq B$.\\[6pt]
\phantom{}\hs\hs iv) $\{Jx_{n}\}_{n=1}^{k}$ is a frame for the  Hilbert space $\left(\mathcal{K}, [\cdot,\cdot]_{J}\right)$ with frame bounds $A\leq B$.
\end{theorem}
\begin{corollary}\label{generadokreinframe}
  Let $(\K,[\cdot,\cdot])$ be a finite-dimensional Pontryagin space with fundamental symmetry $J$, and $\{x_{n}\}_{n=1}^{k}$ be a family of vectors in $\K$. Then the following are equivalent:
  \begin{enumerate}
    \item [$i).$] $\{x_{n}\}_{n=1}^{k}$ is a frame for the Pontryagin space $\K$
    \item [$ii).$] $\{x_{n}\}_{n=1}^{k}$ is a frame for the Hilbert space $(\K,[\cdot,\cdot]_{J})$
    \item [$iii).$] $\operatorname{span}\{x_{n}\}_{n=1}^{k}=\K$.
  \end{enumerate}
\end{corollary}
\begin{proof}
  The equivalence between $i$ and $ii$ is follows from Theorem \ref{prop1}, and the equivalence between $ii$ and $iii$ can be found in \cite{Deguang}.
\end{proof}

\begin{definition}\label{preframeoperator}\cite{KEFER}
 Let $(\mathcal{K}, [\cdot,\cdot])$ be a  finite-dimensional Pontryagin space  with fundamental symmetry $J$ and let
 $(\mathfrak{C}^{k}, [\cdot,\cdot]_{\Complex^{k}})$ be the Pontryagin space with fundamental symmetry $\tilde J$
  given in \eqref{Jcomplexk}.  Given a frame $\{x_{n}\}_{n=1}^{k}$ for $\mathcal{K}$,
 the linear map
 \[ \label{T}
 T: \mathfrak{C}^{k} \To \mathcal{K},\quad T\hs (\alpha_n)_{n=1}^{k} = \sum_{n=1}^{k} \alpha_n k_n
 \]
 is called \emph{pre-frame operator}.
\end{definition}
\begin{remark}
  The pre-frame operator given in \eqref{T} is defined for more generally Krein spaces, it can be found in \cite{KEFER}. Just remember that $\ell_{2}\left(\left\{1,\ldots,k\right\}\right)$ can be identified with $\Complex^{k}$.
\end{remark}
\begin{proposition}(\cite{KEFER})\label{preframeboundedsurjective}
Let $\left(\K, [\cdot,\cdot]\right)$ be a Krein space. The family $\{x_{n}\}_{n\in\N}$
is a frame for $\mathcal{K}$, if and only if $T$ is well defined (i.e.\ bounded)  and surjective.
\end{proposition}

\begin{definition}
Let $\left(\K, [\cdot,\cdot]\right)$ be a  finite-dimensional Pontryagin space . The adjoint of pre-frame operator $T$ is given by
\[  \label{T*}
T^{[*]}x= \tilde{J}([x_n,x])_{n=1}^{k}, \quad k\in\mathcal{K}.
\]
And is so-called \emph{analysis operator}.
\end{definition}
\begin{definition} \label{defS}
Let $(\mathcal{K}, [\cdot,\cdot])$ be a  finite-dimensional Pontryagin space  with fundamental symmetry $J$,
 $(\mathfrak{C}^{k}, [\cdot,\cdot]_{\Complex^{k}})$ a Pontryagin space with fundamental symmetry $\tilde J$
 given in  \ref{examplecomplexk}, and \eqref{Jcomplexk} respectively,
 and $\{k_{n}\}_{n=1}^{k}\subset\mathcal{K}$  a frame for $\mathcal{K}$.
The operator
$$
S:= T\hs\tilde J\hs T^{[*]}
$$
is called \emph{frame operator}.
\end{definition}
It follows immediately from \eqref{T}, \eqref{T*} and $\tilde{J}^2= \mathrm{id}$ that
\[    \label{Sk}
 Sx = \sum_{n\in\N} [x_n\hs,x]\hs x_n,\quad x\in\mathcal{K},
\]
as desired.
Moreover, $S$ is clearly self-adjoint.
If $(\mathfrak{C}^{k}, [\cdot,\cdot])= (\Complex^{k}, \ip{\cdot})$, then $S=TT^*$, exactly as in the
Hilbert space case.

\section{Main Results}\label{resultadosprincipales}

\subsection{Construction of frames with a operator on finite-dimensional Pontryagin space}
\begin{proposition}\label{construccion}
  Let $(\K,[\cdot,\cdot])$ be a $N$-dimensional Pontryagin space with fundamental symmetry $J$, and $S_{0}$ be a $J$-self-adjoint and positive operator with respect to $[\cdot,\cdot]_{J}$. Let $\lambda_{1}\geq\lambda_{2}\geq\ldots\geq\lambda_{N}>0$ be the eigenvalues of $S_{0}$. Fix $k\geq\,N$ and real numbers $a_{1}\geq\,a_{2}\geq\dots\geq\,a_{k}>0$. The following are equivalent:
  \begin{enumerate}
    \item [$i).$] For every $1\leq\,j\leq\,N$
    \begin{equation}
      \sum_{n=1}^{j}a_{n}^{2}\leq\sum_{n=1}^{j}\lambda_{n},\quad\sum_{n=1}^{k}a_{n}^{2}\leq\sum_{n=1}^{N}\lambda_{n}.
    \end{equation}
    \item [$ii).$] There is a frame $\{x_{n}\}_{n=1}^{k}$ for the Hilbert space $(\K,[\cdot,\cdot]_{J})$ with frame operator $S_{0}$ and $\|x_{n}\|_{J}=a_{n}$, for all $n=1,\ldots,k$.
     \item   [$iii).$] There is a frame $\{x_{n}\}_{n=1}^{k}$ for the Pontryagin space $\K$ with frame operator $S_{0}J$ and\linebreak $\|x_{n}\|_{J}=a_{n}$, for all $n=1,\ldots,k$.
     \item [$iv).$] There is a frame $\{Jx_{n}\}_{n=1}^{k}$ for the the Pontryagin space $\K$ with frame operator $JS_{0}$ and $\|x_{n}\|_{J}=a_{n}$, for all $n=1,\ldots,k$.
     \item [$v).$] There is a frame $\{Jx_{n}\}_{n=1}^{k}$ for the Hilbert space $(\K,[\cdot,\cdot]_{J})$ with frame operator $JS_{0}J$ and $\|x_{n}\|_{J}=a_{n}$, for all $n=1,\ldots,k$.
  \end{enumerate}
\end{proposition}
\begin{proof}
  The equivalence between $i$ and $ii$ is proved in \cite{CasazzaLeon}. By Theorem \ref{prop1} we have the equivalences $ii\leftrightarrow\,iii\leftrightarrow\,iv\leftrightarrow\,v$. Only we need to find the relationship of its respective frame operator with the operator $S_{0}$. The details is in \cite{KEFER}.

  Since $S_{0}$  is the frame operator corresponding to the family $\{x_{n}\}_{n=1}^{k}$, then
\begin{align*}
&  S_{1}x=\sum_{n=1}^{k}[x,x_{n}]x_{n}=\sum_{n=1}^{k}[Jx,x_{n}]_{J}x_{n}=S_{0}Jx,\quad\forall x\in\K.
\intertext{$S_{1}=S_{0}J$ is the frame operator for the frame $\{x_{n}\}_{n=1}^{k}$ in the Pontryagin space $\K$.  }
&  S_{2}x=\sum_{n=1}^{k}[x,Jx_{n}]Jx_{n}=J\left(\sum_{n=1}^{k}[x,x_{n}]_{J}x_{n} \right)=JS_{0}x,\quad\forall x\in\K.
\intertext{$S_{2}=JS_{0}$ is the frame operator for the frame $\{Jx_{n}\}_{n=1}^{k}$ in the Pontryagin space $\K$.  }
&  S_{3}x=\sum_{n=1}^{k}[x,Jx_{n}]_{J}Jx_{n}=J\left(\sum_{n=1}^{k}[Jx,x_{n}]_{J}x_{n}\right)=JS_{0}Jx,\quad\forall x\in\K.
\intertext{$S_{3}=JS_{0}J$ is the frame operator for the frame $\{Jx_{n}\}_{n=1}^{k}$ in the Hilbert space $(\K,[\cdot,\cdot]_{J})$.  }
\end{align*}
\end{proof}

\subsection{Extension of frames on finite-dimensional Pontryagin space}\mbox{}\\
\label{extension}
This subsection is based on the study in Hilbert spaces , more details see \cite{Deguang}.
\subsubsection*{Similar Frames in  finite-dimensional Pontryagin  Spaces}
\begin{definition}
   Let $(\h,[\cdot,\cdot]_{\h})$ and $(\K,[\cdot,\cdot]_{\K})$ be  finite-dimensional  Pontryagin spaces.  Two frames $\{x_{n}\}_{n=1}^{k}$ and $\{y_{n}\}_{n=1}^{k}$ for $\h$ and $\K$ respectively, are said to be \emph{similar} if there exists an invertible operator $U:\h\rightarrow\K$ such that $Uy_{n}=x_{n}$  for $n\in\{1,\ldots,k\}$. The frames are called \emph{unitarily equivalent} if we require  $U$ to be a unitary operator from $\h$ onto $\K$.
\end{definition}
\begin{proposition}\label{similarframesth}
 Let $(\h,[\cdot,\cdot]_{\h})$, $(\K,[\cdot,\cdot]_{\K})$ be  a finite-dimensional Pontryagin spaces, and   $\{x_{n}\}_{n=1}^{k}$, $\{y_{n}\}_{n=1}^{k}$ be two frames for $\h$ and $\K$ respectively. Then they are similar  if and only if their analysis operators have the same range.
\end{proposition}

\begin{proof}
Let $T_{\h}^{[*]_{\h}}$ and $T_{\K}^{[*]_{\K}}$ be the analysis operator for $\{x_{n}\}_{n=1}^{k}$ y $\{y_{n}\}_{n=1}^{k}$ respectively.

  $\Rightarrow]$ Suppose that $\{x_{n}\}_{n=1}^{k}$ y $\{y_{n}\}_{n=1}^{k}$ are similar, then there is an invertible operator $U:\h_{N}\rightarrow\K_{M}$ such that $Ux_{n}=y_{n}$. Hence
  \begin{align*}
    T_{\K}^{[*]_{\K}}y&=\sum_{j=1}^{k}(-1)^{j-1}[y,y_{j}]e_{j}=\sum_{j=1}^{k}(-1)^{j-1}[y,Ux_{j}]e_{j}=\sum_{j=1}^{k}(-1)^{j-1}[U^{*}y,x_{j}]e_{j}\\
    &=T_{\h}^{[*]_{\h}}\left(U^{*}y\right).
  \end{align*}
  Now, since $U^{*}$ is invertible if and only if $U$ is, we concluded that $$T_{\K}^{[*]_{\K}}(\K)=T_{\h}^{[*]_{\h}}U^{*}(\K)=T_{\h}^{[*]_{\h}}(\h).$$

  $[\Leftarrow$ Suposse that $T_{\K}^{[*]_{\K}}(\K)=T_{\h}^{[*]_{\h}}(\h)=V$. We note that $\left.T_{\h}\right|_{V}$ y $\left.T_{\K}\right|_{V}$ are invertible, thus the operator $G=T_{\K}\left(T_{\h}^{[*]_{\K}}T_{\h}\right)^{-1}T_{\h}^{[*]_{\h}}:\h\rightarrow\K$ is well defined and it is invertible. Let $P$ be a $J$-ortogonal projection of $\Complex^{k}$ on $V$. If $y\in\,V^{[\perp]_{\Complex^{k}}}$, then $0=[x,y]_{\Complex^{k}}$ for all $x\in\,V$. Thus
$$0=[x,y]_{\Complex^{k}}=[T_{\K}^{[*]_{\K}}z,y]_{\Complex^{k}}=[z,T_{\K}y]_{\Complex^{k}},$$
Since $y$ is arbitrary, we concluded that $T_{\K}\left(V^{[\perp]_{\Complex^{k}}}\right)=\{0\}$. Similary we have that $T_{\h}\left(V^{[\perp]_{\Complex^{k}}}\right)=\{0\}$. Thus, by the definition of pre-frame operator we consider $v_{m,j}=\delta_{m,j}\in\Complex^{k}$, and we get $$T_{\h}Pv_{m}=T_{\h}v_{m}=\sum_{j=1}^{k}\delta_{m,j}x_{m}=x_{m},$$
on the other hand, it also satisfies that $y_{m}=T_{\K}v_{m}=T_{\K}Pv_{m}$. Hence,
\begin{align}
  Ux_{m}=UT_{\h}Pv_{m}=T_{\K}\left(T_{\h}^{[*]_{\h}}T_{\h}\right)^{-1}T_{\h}^{[*]_{\h}}T_{\h}Pv_{m}=T_{\K}Pv_{m}=y_{m},
\end{align}
This implies that the frames are similar.
\end{proof}

\begin{theorem}\label{levantamiento2}
  Let $(\h,[\cdot,\cdot]_{\h})$ be a Pontryagin space, and  suppose that $\{x_{n}\}_{n=1}^{k}$ is frame for $\h$. Then there exists a Pontryagin space $(\K,[\cdot,\cdot]_{\K})$ with $\h\subset\K$, and a frame $\{v_{n}\}_{n=1}^{k}$ for $\K$ such that $x_{n}=Pv_{n}$, where  $P$ is the $J$-orthogonal projection from $\K$ onto $\h$.
\end{theorem}

\begin{proof}
 Consider  $V=\left(\operatorname{Rang}\,T^{[*]}\right)^{[\perp]}\subset\mathfrak{C}^{k}$, which is a Pontryagin space. Hence by Proposition \ref{proposicionproductokrein} we have that $\K=\h_{N}\times\,V=\K\times\{0\}\,\oplus\,\{0\}\times\,V\simeq \K\,\oplus\,V$  is a Pontryagin space with indefinite inner product $[\cdot,\cdot]_{\K}=[\cdot,\cdot]_{\h}+[\cdot,\cdot]_{\Complex^{k}}$, and fundamental symmetry
 $J_{\K}=J_{\h_{N}}\oplus\,\widetilde{J},$ where $\widetilde{J}$ is the fundamental simmetry of Pontryagin space $\mathfrak{C}^{k}$ given to \eqref{Jcomplexk}.  Consider $Q^{[\perp]}=\mathrm{I}-Q$, where $Q:\mathfrak{C}^{k}\rightarrow\operatorname{Rang}T^{[*]}$ is a $J$-orthogonal projection  from $\mathfrak{C}^{k}$ onto $\operatorname{Rang}T^{[*]}$. The family $\{e_{n}\}_{n=1}^{k}$, where $e_{n}$ is the vector which has a $1$ in the $n$-esimo place and zero in the other, it is a base $J$-orthonormal from $\mathfrak{C}^{k}$, and by \eqref{prodindecomplex} it satisfies that $[e_{n},e_{n}]_{\Complex^{k}}=(-1)^{n-1}$.

 We defined
  \begin{equation}
  u_{n}=x_{n}\oplus\,Q^{[\perp]}e_{n},\quad\,n=1,2,\ldots,k.
  \end{equation}
  Note that $P:\K_{M}\rightarrow\h_{N}$ is a  $J$-orthogonal projection from $\K$ into $\h$, we will have that $Pu_{n}=x_{n}$. It is enough to show that the family $\{u_{n}\}_{n=1}^{k}$ is a frame for $\K_{M}$. Indeed, note that by corollary \ref{generadokreinframe} the family $\{e_{n}\}_{n=1}^{k}$  is a frame for $\mathfrak{C}^{k}$,  therefore $\{Qe_{n}\}_{n=1}^{k}$ is a frame for $\operatorname{Rang}T^{[*]}$, (See \cite{KEFER}).  Given $x\in\operatorname{Rang}T^{[*]}$, by \eqref{prodindecomplex} and \eqref{teoremabasekrein} it is obtained that
  \begin{align*}
T_{Q}^{[*]_{\Complex^{k}}}(x)&=\sum_{n=1}^{k}(-1)^{n-1}[x,Qe_{n}]_{\Complex^{k}}e_{n}=\sum_{n=1}^{k}(-1)^{n-1}[Qx,e_{n}]_{\Complex^{k}}e_{n}=\sum_{n=1}^{k}(-1)^{n-1}[x,e_{n}]_{\Complex^{k}}e_{n}\\
&=\sum_{n=1}^{k}[e_{n},e_{n}][x,e_{n}]_{\Complex^{k}}e_{n}=x,
\end{align*}

Where $T_{Q}$ is the pre-frame operator of $\{Qe_{n}\}_{n=1}^{k}$. That is, $\operatorname{Rang}T_{Q}^{[*]_{\Complex^{k}}}=\operatorname{Rang}T^{[*]}$, by theorem \ref{similarframesth}, the frames $\{x_{n}\}_{n=1}^{k}$ y $\{Qe_{n}\}_{n=1}^{k}$ are similars. Let $W$ be the invertible operator such that $WQe_{n}=x_{n}$ for $n=1,\ldots,k$. Therefore, for each $n$ we have that
$$u_{n}=WQe_{n}\oplus\,Q^{[\perp]}e_{n}=U\left(Qe_{n}\oplus\,Q^{[\perp]}e_{n}\right)=Ue_{n},$$
where $U=W\oplus\mathrm{I}$ is an invertible operator from $\mathfrak{C}^{k}$ into $\K_{M}$. Furthermore, since $\{e_{n}\}_{n=1}^{k}$  is a frame for $\mathfrak{C}^{k}$ and the invertible operators preserve frames (ver \cite{KEFER}), thus it concludes that $\{u_{n}\}_{n=1}^{k}$ is a frame for $\K_{M}$.
\end{proof}

\subsection{Coupling of frames on finite-dimensonal Pontryagin space}
\label{acoplamiento}

\begin{theorem}
  Let $\left(\K,[\cdot,\cdot]_{\K}\right)$, and $\left(\h,[\cdot,\cdot]_{\h}\right)$ be Pontryagin spaces. If $\{x_{n}\}_{n=1}^{k}$ and $\{y_{n}\}_{n=1}^{k}$ are frames for $\K$ and $\h$ respectively, then there is a Pontryagin space $\Re$ containing $\K,$ and $\h$, a frame $\{z_{n}\}_{n=1}^{k}$ for $\Re$ such that $P_{\K}z_{n}=x_{n}$ and $P_{\h}z_{n}=y_{n}$, for each $n=1,\ldots,k$, where $P_{\K}$ and $P_{\h}$ are $J$-orthogonal projection from $\Re$ onto $\K$ and $\h$ respectively.
\end{theorem}
\begin{proof}
  By Theorem \ref{levantamiento2} there are Pontryagin spaces $\Re_{\K}$ and $\Re_{\h}$, such that $\K\subset\Re_{\K}$, $\h\subset\Re_{\h}$, frames $\{u_{n}\}_{n=1}^{k}$, $\{v_{n}\}_{n=1}^{k}$ for $\Re_{\K}$ and $\Re_{\h}$ respectively, with $P_{1}u_{n}=x_{n}$, $P_{2}v_{n}=y_{n}$  for each $n=1,\ldots,k$. Here $P_{1}:\Re_{\K}\To\K$, $P_{2}:\Re_{\h}\To\h$  are $J$-orthogonal projections. Define $\Re=\Re_{1}\times\Re_{2}$, by Proposition \ref{proposicionproductokrein} is a Pontryagin space with respect to the indefinite inner product $[\cdot,\cdot]_{\Re}=[\cdot,\cdot]_{\Re_{\K}}+[\cdot,\cdot]_{\Re_{\h}}$, where $\K\times\h\subset\Re$. If  $T_{\Re_{\K}}$ is the pre-frame operator of the family $\{u_{n}\}_{n=1}^{k}$ and $T_{\Re_{\h}}$ is  the pre-frame operator  of the family $\{v_{n}\}_{n=1}^{k}$, then,  taking  $0<B=\|T_{\Re_{\K}}\|+\|T_{\Re_{\h}}\|<\infty$ and $T_{\Re}:\mathfrak{C}^{k}\To\Re$, by
  $$T(\{c_{n}\}_{n=1}^{k})=\sum_{n=1}^{k}c_{n}(u_{n},v_{n})=\left(\sum_{n=1}^{k}c_{n}u_{n},\sum_{n=1}^{k}c_{n}v_{n}\right)=\left(T_{\Re_{\K}}(\{c_{n}\}_{n=1}^{k}),T_{\Re_{\h}}(\{c_{n}\}_{n=1}^{k})\right),$$
  we get
  \begin{align*}
   \|T_{\Re}(\{c_{n}\}_{n=1}^{k})\|_{J_{\Re}}^{2}&=\left[\left(T_{\Re_{\K}}(\{c_{n}\}_{n=1}^{k}),T_{\Re_{\h}}(\{c_{n}\}_{n=1}^{k})\right),\left(T_{\Re_{\K}}(\{c_{n}\}_{n=1}^{k}),T_{\Re_{\h}}(\{c_{n}\}_{n=1}^{k})\right)\right]_{J_{\Re}}\\
   &=\left[T_{\Re_{\K}}(\{c_{n}\}_{n=1}^{k}),T_{\Re_{\K}}(\{c_{n}\}_{n=1}^{k})\right]_{J_{\Re_{\K}}}+\left[T_{\Re_{\h}}(\{c_{n}\}_{n=1}^{k}),T_{\Re_{\h}}(\{c_{n}\}_{n=1}^{k})\right]_{J_{\Re_{\h}}}\\
   &=\|T_{\Re_{\K}}(\{c_{n}\}_{n=1}^{k})\|_{J_{\Re_{\K}}}^{2}+\|T_{\Re_{\h}}(\{c_{n}\}_{n=1}^{k})\|_{J_{\Re_{\h}}}^{2}\\
   &\leq B\,\|\{c_{n}\}_{n=1}^{k}\|_{J_{\Complex^{k}}}^{2}.
  \end{align*}
  Hence, $\|T_{\Re}\|\leq\,\sqrt{B}$. Furthermore, $T_{\Re}$ is surjective.  Thus, by Proposition \ref{preframeboundedsurjective}\linebreak the family  $\{z_{n}=(u_{n},v_{n})\}_{n=1}^{k}$  is a frame for $\Re$.  On the other hand, define the linear operator $P:\Re\To\K\times\h$ given by the formula $P(x,y)=(P_{1}x,P_{2}y)$, note that
  \begin{align*}
[P(x,y),(a,b)]_{\Re}&=[(P_{1}x,P_{2}y),(a,b)]_{\Re}=[P_{1}x,a]_{\Re_{\K}}+[P_{2}y,b]_{\Re_{\h}}=[x,P_{1}a]_{\Re_{\K}}+[y,P_{2}b]_{\Re_{\h}}\\
&=[(x,y),P(a,b)]_{\Re}.
\end{align*}
  Thus, $P^{[*]}=P$. Now, $P^{2}(x,y)=P(P_{1}x,P_{2}y)=(P_{1}^{2}x,P_{2}^{2}y)=(P_{1}x,P_{2}y)=P(x,y)$ para todo $(x,y)\in\Re$. Esto es, $P^{2}=P$, for which $P$ is a $J$-orthogonal projection from $\Re$ onto $\K\times\h$ and satisfies $Pz_{n}=(x_{n},y_{n})$ for each $n=1,\ldots,k$. We finish taking $P_{\K}(x,y)=(P_{1}x,0)$ and $P_{\h}(x,y)=(0,P_{2}y)$.
\end{proof}

\begin{proposition} [\bf{Couplers for frame operators}] \label{acopladorS} Let $(\K,[\cdot,\cdot]_{\K})$ and $(\h,[\cdot,\cdot]_{\h})$  be  finite-dimensional Pontryagin spaces with fundamental symmetries $J_{\K}$ and $J_{\h}$ respectively. Let $S_{\K}$ and $S_{\h}$ be  $J$-self-adjoint and positive operators with respect to $[\cdot,\cdot]_{J_{\K}}$, and $[\cdot,\cdot]_{J_{\h}}$ resp. Then, there exists a  finite-dimensional Pontryagin space $\Re$ with fundamental symmetry $J$, and a $J$-self-adjoint and positive operator $S_{\Re}:\Re\rightarrow\Re$, such that:
\begin{equation}
\left. S_{\Re}\right|_{\K}=S_{\K},\quad\text{and}\quad  \left.S_{\Re}\right|_{\h}=S_{\h}.
\end{equation}
\end{proposition}
\begin{proof}
By Proposition \ref{proposicionproductokrein}  $\Re=\K\times\h$ is a Pontryagin space with indefinite inner product  $[(x,y),(a,b)]_{\Re}=[x,a]_{\K}+[y,b]_{\h}$ and fundamental symmetry given by the formula \linebreak$J_{\Re}(x,y)=\left(J_{\mathcal{K}}x,J_{\h}y\right), \,\forall\,(x,y)\in\Re.$
Define $S_{\Re}:\Re\To\Re$ by
\begin{equation}
S_{\Re}(x,y)=(S_{\K}x,S_{\h}y), \quad (x,y)\in \Re.
\end{equation}
This operator is positive and $J_{\Re}$-self-adjoint with respect to $[\cdot,\cdot]_{J_{\Re}}$. Furthermore
$$\left. S_{\Re}\right|_{\K}=(S_{\K},0),\quad\text{and}\quad  \left.S_{\Re}\right|_{\h}=(0,S_{\h}).$$
\end{proof}

\begin{remark}
If the operators $S_{\K}$ y $S_{\h}$ given at the proposition \ref{acopladorS} satisfy the conditions at the proposition \ref{construccion}, then the operator $S_{\Re}$
has an   associated frame and satisfies the properties given for the theorem \ref{levantamiento2}.
\end{remark}
\section*{Acknowledgements}
  The second author thanks the coauthors and to Ricardo Cede\~no for their hospitality during the stay in the city of Neiva. The second author is supported  by the pockets of co-authors and of course  by the Universidad Surcolombiana.

% ------------------------------------------------------------------------

%\subsection*{Acknowledgment}
%Many thanks to our \TeX-pert for developing this class file.
% ------------------------------------------------------------------------
\end{document}